\RequirePackage{fix-cm}
\documentclass[11pt,a4paper]{article}
\usepackage{graphicx}
\usepackage{tikz}
\usepackage{amsmath}
\usepackage{amssymb}
\usepackage{amsthm}
\usepackage{amsfonts}
\usepackage{url}
\usepackage{hyperref}
\usetikzlibrary{arrows.meta}
\newcommand{\bint}[1]{\boldsymbol{#1}}
\newcommand{\lb}[1]{\underline{#1}}
\newcommand{\ub}[1]{\overline{#1}}
\newtheorem{theorem}{Theorem}

\theoremstyle{definition}

\newtheorem{example}{Example}

\begin{document}
	
	\title{Interval Linear Programming under Transformations: Optimal Solutions and Optimal Value Range\thanks{The first two authors were supported by the Czech Science Foundation under Grant P403-18-04735S and by the Charles University, project GA UK No. 156317. The third author was supported by the Czech Science Foundation under Grant P403-17-13086S.}
	}
	
\author{
        Elif Garajov\'a\footnote{
Charles University, Faculty  of  Mathematics  and  Physics,
Department of Applied Mathematics, 
Malostransk\'e n\'am.~25, 11800, Prague, Czech Republic, 
e-mail: \texttt{elif@kam.mff.cuni.cz}
}
\and
        Milan Hlad\'ik\footnote{
Charles University, Faculty  of  Mathematics  and  Physics,
Department of Applied Mathematics, 
Malostransk\'e n\'am.~25, 11800, Prague, Czech Republic, 
e-mail: \texttt{hladik@kam.mff.cuni.cz}
}
\and
	Miroslav Rada\footnote{
University of Economics, Faculty of Finance and Accounting, Department of Financial Accounting and Auditing, W. Churchill's Sq. 4, 130 67 Prague, Czech Republic,
e-mail: \texttt{miroslav.rada@vse.cz}
}
}

\maketitle
	
	\begin{abstract}
		Interval linear programming provides a tool for solving real-world optimization problems under interval-valued uncertainty. Instead of approximating or estimating crisp input data, the coefficients of an interval program may perturb independently within the given lower and upper bounds. However, contrarily to classical linear programming, an interval program cannot always be converted into a desired form without affecting its properties, due to the so-called dependency problem.
		
		In this paper, we discuss the common transformations used in linear programming, such as imposing non-negativity on free variables or splitting equations into inequalities, and their effects on interval programs. Specifically, we examine changes in the set of all optimal solutions, optimal values and the optimal value range. Since some of the considered properties do not holds in the general case, we also study a special class of interval programs, in which uncertainty only affects the objective function and the right-hand-side vector. For this class, we obtain stronger results. 
	\end{abstract}
	
\emph{Keywords.} Interval linear programming; Optimal set; Optimal value range; Transformations

	%% For one-column wide figures use
	%\begin{figure}
	%% Use the relevant command to insert your figure file.
	%% For example, with the graphicx package use
	%  \includegraphics{example.eps}
	%\end{figure}
	%%
	%% For two-column wide figures use
	%\begin{figure*}
	%% Use the relevant command to insert your figure file.
	%% For example, with the graphicx package use
	%  \includegraphics[width=0.75\textwidth]{example.eps}
	%\end{figure*}
	%\begin{table}
	%\begin{tabular}{lll}
	%\hline\noalign{\smallskip}
	%\end{tabular}
	%\end{table}
	
	\section{Introduction}\label{sec:intro}
	
	When handling real-world optimization problems by means of mathematical modeling, it is often necessary to treat inexact or uncertain input data due to various measurement errors or estimations. Throughout the years, several approaches for solving optimization problems with imprecise data have emerged, based on different sources of uncertainty and different requirements imposed on the solutions. This paper adopts the approach of interval linear programming, which can be considered a special case of multi-parametric optimization with no dependencies among the coefficients. Interval-valued coefficients have been used for modeling uncertainty in many practical applications, such as portfolio selection \cite{Lai:2002, Kumar:2016}, environmental management \cite{Tan:2010, Cheng:2015} or group decision making \cite{Groselj:2017}.
	
	In interval optimization, it is assumed that the coefficients may perturb independently within the given lower and upper bounds. However, rather than focusing on the worst case and trying to find a stable solution (as in robust optimization), we aim to cover the properties of all possible scenarios. This approach leads to questions such as describing the set of feasible solutions over all scenarios \cite{Oettli:1964, Rohn:InexactLP:2006} and the set of all optimal solutions \cite{Allahdadi:2013:opt}, computing the optimal value range \cite{Mraz:1998, Rohn:InexactLP:2006b, Hladik:2009} or characterizing the duality gaps \cite{Novotna:2017}. For an overview of other results in interval linear programming see the survey by \cite{Hladik:2012}.
	
	This paper discusses an ever-present issue in interval optimization and interval analysis in general: \emph{the dependency problem}. Since we do not allow any dependencies among the coefficients of an interval program and assume that the values of the coefficients perturb independently, any transformation that changes the program and duplicates some of the interval coefficients will also create new possible scenarios. Due to the dependency problem, transforming an interval program into a desired form by applying the standard transformations used in classical linear programming may change its feasible or optimal solutions and other important properties. Therefore, it may be necessary to study different types of programs separately, since the results obtained for one type do not have to hold for other types.
	
	We examine the effects of the transformations on the optimal solution set and the optimal value range in order to identify the transformations that are also applicable to interval linear programs while preserving the considered properties. This will allow us to directly generalize the results known for one type of programs to other forms, as well as highlight the cases in which the particular properties may differ. We will also study the problem on a special class of interval linear programs with interval coefficients only appearing in the objective function and the right-hand-side vector in order to derive stronger results. These types of programs arise in applications, where the coefficient matrix represents some known relations, such as transportation problems \cite{Safi:2013, Cerulli:2017:inttransp} or network flow problems \cite{Hashemi:2006}.
	
	The paper is structured as follows: Section~\ref{sec:ilp} introduces the basic notions and terminology of interval linear programming. In Section~\ref{sec:transf}, we review the standard transformations used in classical linear programming and discuss their applicability on interval programs. Sections~\ref{sec:optsol} and~\ref{sec:optval} explore the effects of the transformations on the set of all optimal solutions and optimal values of an interval linear program. Section~\ref{sec:concl} summarizes and concludes the paper. 
	
	\section{Interval Linear Programming}\label{sec:ilp}
	Given two real matrices $\lb{A}, \ub{A} \in \mathbb{R}^{m\times n}$ satisfying $\lb{A} \le \ub{A}$, we define an \emph{interval matrix} as the set \[\bint{A} = [\lb{A}, \ub{A}] = \{A \in \mathbb{R}^{m \times n} : \lb{A} \le A \le \ub{A}\}.\]
	The matrices $\lb{A}, \ub{A}$ are the \emph{lower} and \emph{upper bound} of $\bint{A}$, respectively. Alternatively, we can also define an interval matrix by its \emph{center} $A_c = \frac{1}{2}(\ub{A} + \lb{A})$ and \emph{radius} $A_\Delta = \frac{1}{2}(\ub{A} - \lb{A})$.
	Throughout the paper, we denote interval objects by bold lowercase (for one-dimensional intervals and interval vectors) and bold uppercase (for interval matrices) letters. The symbol $\mathbb{IR}$ is used to denote the set of all (real) intervals.
	
	Furthermore, let us extend the classical definition of linear programs to the case of uncertain interval coefficients. We define an \emph{interval linear program} (ILP) as a family of linear programs
	\begin{equation}\label{eq:ilp:def}
	\{ \text{minimize } c^T x 
	\text{ subject to} \  x \in \mathcal{M}(A,b) : A \in \bint{A}, b \in \bint{b}, c \in \bint{c} \}
	\end{equation}
	where $\mathcal{M}(A,b)$ denotes the feasible set described by some linear constraints. For the sake of brevity, we usually write ILPs in the form of linear programs with interval coefficients (see \eqref{eq:ilp:A}--\eqref{eq:ilp:C} below for examples). A particular linear program in the family is called a~\emph{scenario}. Given an interval linear program, we can define its \emph{dual} ILP as the family of all dual linear programs.
	
	Similarly, we can also generalize the notion of feasibility and optimality to the interval case. We say that a vector $x \in \mathbb{R}^n$ is a \emph{(weakly) feasible/optimal} solution to the interval linear program, if there exists a scenario determined by the triplet $A \in \bint{A}$, $b \in \bint{b}, c \in \bint{c}$ such that $x$ is a feasible/optimal solution for the scenario (see \cite{Rohn:InexactLP:2006} and \cite{Hladik:2017} for an overview of the feasibility properties of interval linear systems). 
	The \emph{(weakly) feasible/optimal set} is the set of all weakly feasible/optimal solutions of an ILP.
	We often omit the word ``weakly'', since no confusion should arise in the paper.
	
	Considering the optimal objective values, we may be interested in the best and the worst value that can be achieved as optimal for some scenario of the given ILP. For a minimization program, these bounds are defined as 
	\begin{align*}
	&\lb{f}(\bint{A}, \bint{b}, \bint{c}) = \inf \{f(A,b,c): A \in \bint{A}, b \in \bint{b}, c \in \bint{c}\},\\
	&\ub{f}(\bint{A}, \bint{b}, \bint{c}) = \sup \{f(A,b,c): A \in \bint{A}, b \in \bint{b}, c \in \bint{c}\},
	\end{align*}
	where $f(A,b,c)$ denotes the optimal value of the corresponding linear program, allowing the values $-\infty$ and $\infty$ for unbounded programs or infeasible programs.
	The interval $[\lb{f}(\bint{A}, \bint{b}, \bint{c}), \ub{f}(\bint{A}, \bint{b}, \bint{c})]$ is called the \emph{optimal value range}.
	
	Due to the dependency problem, we usually assume that an~ILP is given either in the general form
	\begin{equation*}
	\begin{array}{lrrrr@{\ }l}
	\text{minimize } & \multicolumn{5}{l}{\bint{c_1}^T x_1 + \bint{c_2}^T x_2} \\
	\text{subject to } & \bint{A}x_1 &+ & \bint{B}x_2 & = & \bint{b_1},\\
	&  \bint{C}x_1 & + & \bint{D}x_2 & \le & \bint{b_2},\\
	& &  \multicolumn{2}{r}{x_1, x_2} &\ge & 0,
	\end{array}
	\end{equation*}
	or in one of the following forms often encountered in optimization problems:
	\begin{align}
	&\text{minimize } \bint{c}^T x \text{ subject to } \bint{A}x = \bint{b},\;x \ge 0, \tag{I}\label{eq:ilp:A}\\ 
	&\text{minimize } \bint{c}^T x \text{ subject to } \bint{A}x \le \bint{b}, \tag{II}\label{eq:ilp:B}\\
	&\text{minimize } \bint{c}^T x \text{ subject to } \bint{A}x \le \bint{b},\;x \ge 0. \tag{III}\label{eq:ilp:C}
	\end{align}
	
	\section{Transformations of Interval Linear Programs}\label{sec:transf}
	Let us now briefly review the basic transformations that are used in linear programming to convert a given program into the desired form. While some of them can also be applied to interval programs, other transformations may change the set of feasible or optimal solutions and affect the properties of the program. The effects of these transformations on the set of all optimal solutions and optimal values will be discussed in more detail in Sections~\ref{sec:optsol} and~\ref{sec:optval}.
	
	\paragraph{Changing the objective.} A commonly used trick in linear programming is to switch the objective from maximization to minimization (or vice versa) by taking the opposite of the objective function. Since this transformation does not duplicate any coefficients, it can also be directly applied to interval programs: an interval objective $\max \bint{c}^T x$ can be equivalently rewritten as $- (\min -\bint{c}^T x)$.
	
	\paragraph{Adding slack variables.} In order to convert inequality constraints to equations in a linear program, sign-restricted slack variables are added to the constraints. Analogously, we can transform an interval inequality system $\bint{A}x \le \bint{b}$ into an interval system \[ \bint{A}x + Iy = \bint{b},\; y \ge 0. \]
	Again, the transformation does not duplicate any of the existing coefficients and only introduces some new crisp coefficients. Applying this transformation to a set of constraints results in an equivalent interval linear program, since each new scenario is equivalent to a corresponding scenario in the original interval system.
	
	\paragraph{Splitting equations into inequalities.} When changing a general linear program into an inequality-constrained form, we split each equation into two opposite inequalities. However, splitting an interval equation into interval inequalities while preserving all properties is not always possible, due to the dependency problem. Breaking dependencies among multiple occurrences of an interval coefficient leads to generating new scenarios that do not have a counterpart in the original problem. More precisely, while an interval equation $\bint{a}^T x = b$ consists of the scenarios $a^T x = b$ with all possible choices of $a \in \bint{a}$, the system of two interval inequalities $\bint{a}^T x \le b$, $\bint{a}^T x \ge b$ corresponds to the (possibly larger) set of all scenarios in the form $a_1^T x \le b, a_2^T x \ge b$ with coefficients $a_1, a_2 \in \bint{a}$.
	
	\paragraph{Imposing non-negativity.} Another transformation used in linear programming is substituting the difference of two non-negative variables for a free variable. Along with introducing two new variables for each original free variables, the operation also duplicates the corresponding coefficients. This, again, leads to breaking the dependency in case of interval constraints: By replacing a constraint $\bint{a}^T x = \bint{b}$ with $\bint{a}^T x^+ - \bint{a}^T x^- = \bint{b}$ and $x^+, x^- \ge 0$, we create two independent occurrences of the interval coefficient $\bint{a}$, which can take on different values in a particular scenario.
	
	\section{Optimal Solution Set under Transformations}\label{sec:optsol}
	\subsection{General Case}\label{ssec:gen}
	In this section, we discuss the two types of transformations that can possibly change the properties of an interval linear program: splitting equation constraints and imposing non-negativity on the variables. First, we present examples of ILPs showing that these transformations can change the optimal solution set. Note that in both cases, the transformed interval program contains all of the original scenarios, therefore, the original optimal set is always a subset of the transformed one.
	
	We have seen in Section~\ref{sec:transf} that splitting an interval equation into two inequalities can cause a dependency problem and possibly change the feasible or optimal solutions of an ILP. However, it was proved by \cite{Li:2015} that this is not the case for the feasible set.
	
	\begin{theorem}[\cite{Li:2015}]\label{thm:Li}
		A vector $x \in \mathbb{R}^n$ is a weakly feasible solution of an~interval linear system $\bint{A}x = \bint{b}$ if and only if it is a weakly feasible solution of the system $\bint{A}x \le \bint{b}, \bint{A}x \ge \bint{b}$.
	\end{theorem}
	
	Even though the transformation preserves the weakly feasible solution set, the following example shows that creating new scenarios may, in fact, lead to expanding the set of optimal solutions.
	
	\begin{example}\label{ex:optsol}
		Consider the interval linear program
		\begin{equation}\label{eq:ex}
		\begin{array}{lr@{\ }l}
		\text{minimize } & \multicolumn{2}{l}{-x_1} \\
		\text{subject to } & [0,1]x_1 - x_2 = 0,\\
		&  x_2 \le 1,\\
		&  x_1, x_2 \ge 0.
		\end{array}
		\end{equation}
		The optimal solution set of~\eqref{eq:ex} consists of all values with $x_1 \in [1, \infty)$ and $x_2 = 1$. By splitting the equation $[0,1]x_1 - x_2 = 0$ into two inequalities, we allow the value for each coefficient to be chosen independently. This results in introducing additional scenarios into the problem, such as the scenario 
		\begin{equation*}
		1x_1 - x_2 \le 0,\quad 0x_1 - x_2 \ge 0
		\end{equation*}
		with the optimal solution $(0,0)$. Therefore, the optimal solution sets of the two problems are not equivalent.
	\end{example}
	
	The second type of transformation affected by the dependency problem is restricting the sign of the variables. However, unlike all of the other transformations, replacing a free variable by a difference of two non-negative variables does not even preserve the feasible solution set of an interval linear program, as already noted by \cite{Hladik:2012}. This implies that the set of optimal solutions may also change, in general.
	
	\subsection{Special Case: Fixed Coefficient Matrix}\label{ssec:spec}
	This section is devoted to a special class of interval linear programs, in which uncertainty only affects the objective function and the right-hand-side vector. We show that, unlike general interval linear programs, programs with a fixed coefficient matrix allow for all types of transformations while preserving the optimal solution set. Theorem~\ref{thm:spec:AC} presents the result for the transformation of splitting an equation into two opposite inequalities, i.e. converting a program of type~\eqref{eq:ilp:A} with a fixed matrix into a program of type~\eqref{eq:ilp:C}.
	
	\begin{theorem}\label{thm:spec:AC}
		Let $\mathcal{S}(A,\bint{b},\bint{c})$ be the optimal solution set of an interval linear program of type~\eqref{eq:ilp:A} with a fixed coefficient matrix given by the triplet $(A, \bint{b}, \bint{c})$. Then, $\mathcal{S}(A,\bint{b},\bint{c})$ is equal to the optimal solution set of the problem
		\begin{equation}\label{eq:transfAC}
		\begin{array}{l@{\hskip 8pt}r@{\ }l}
		\textnormal{minimize} & \multicolumn{2}{l}{\bint{c}^T x}\\
		\textnormal{subject to} & A x &\le \bint{b}_1,\\
		& -A x &\le -\bint{b}_2,\\
		& x &\ge 0.
		\end{array}
		\end{equation}
		with $\bint{b}_1 = \bint{b}_2 = \bint{b}$.
	\end{theorem}
	\begin{proof}
		Let $\mathcal{S}(A',\bint{b'},\bint{c})$ denote the optimal solution set of program~\eqref{eq:transfAC}. Clearly, the inclusion $\mathcal{S}(A,\bint{b},\bint{c}) \subseteq \mathcal{S}(A',\bint{b'},\bint{c})$ holds, since \eqref{eq:transfAC} contains all the scenarios of \eqref{eq:ilp:A}.
		
		On the other hand, let $x'$ be an optimal solution of ILP \eqref{eq:transfAC} for a scenario determined by the coefficient vectors $c \in \bint{c}$ and $b_1, b_2 \in \bint{b}$ satisfying 
		\begin{equation}\label{eq:transf:scenario}
		Ax' \le b_1,\; -Ax' \le -b_2.
		\end{equation}
		Since we have, $b_2 \le Ax' \le b_1$, there exists $b_3 \in [b_2, b_1] \subseteq \bint{b}$ with $Ax' = b_3$. We claim that $x'$ is also an optimal solution of program~\eqref{eq:ilp:A} for the scenario
		\begin{equation*}
		\begin{array}{lrl@{\ }l}
		\text{minimize} & \multicolumn{2}{l}{c^T x} \\
		\text{subject to} & Ax &=b_3,\\
		& x & \ge 0.
		\end{array}
		\end{equation*}
		Suppose, for the sake of contradiction, that there exists another feasible solution $x^*$ with $c^T x^* < c^T x'$. By the choice of $b_3$, the vector $x^*$ is also feasible for scenario~\eqref{eq:transf:scenario}. However, since the objective function is the same for both problems, this contradicts the assumption that $x'$ is optimal in scenario~\eqref{eq:transf:scenario}.
	\end{proof}
	
	Note that while the transformation preserves the optimal solutions, it may still change other properties of the program, such as the existence of infeasible scenarios or the range of optimal values, which will be discussed in more detail in Section~\ref{sec:optval}. Theorem~\ref{thm:spec:nonneg} shows an analogous result for the transformation of free variables into non-negative variables.
	
	\begin{theorem}\label{thm:spec:nonneg}
		Let $\mathcal{S}(A, \bint{b}, \bint{c})$ denote the optimal solution set of an interval linear program of type~\eqref{eq:ilp:B} with a fixed coefficient matrix and let $\mathcal{S}(A', \bint{b}, \bint{c}')$ be the optimal solution set of the program
		\begin{equation}\label{eq:transfBC}
		\begin{array}{lr@{\ }l}
		\textnormal{minimize} & \multicolumn{2}{l}{\bint{c}_1^T x^+ - \bint{c}_2^T x^-}\\
		\textnormal{subject to} & A x^+ - A x^- &\le \bint{b},\\
		& x^+, x^- &\ge 0
		\end{array}
		\end{equation}
		with $\bint{c}_1 = \bint{c}_2 = \bint{c}$. Then, the following properties hold:
		\begin{enumerate}
			\item If $x \in \mathcal{S}(A,\bint{b},\bint{c})$, then there exists $(x^+, x^-) \in \mathcal{S}(A', \bint{b}, \bint{c}')$ with $x = x^+ - x^-$.
			\item Conversely, if $(x^+, x^-) \in \mathcal{S}(A', \bint{b}, \bint{c}')$, then $x^+ - x^- \in \mathcal{S}(A,\bint{b},\bint{c})$.
		\end{enumerate}
	\end{theorem}
	\begin{proof}
		Let $x \in \mathcal{S}(A,\bint{b},\bint{c})$ be an optimal solution for a scenario $c \in \bint{c}$, $b \in \bint{b}$. Decompose $x$ into the non-negative vectors $x^+ = \max(0,x)$ and $x^- = -\min(0,x)$, where the operations $\max$ and $\min$ are understood entry-wise. The vectors satisfy $x = x^+ - x^-$ and it is easy to see that the pair $(x^+, x^-)$ is optimal in ILP~\eqref{eq:transfBC} for the scenario determined by the objective vector $(c, -c)$ and the right-hand side $b$.
		
		Now, let $(x^+, x^-) \in \mathcal{S}(A', \bint{b}, \bint{c}')$ be optimal for some $b \in \bint{b}$ and $c_1, c_2 \in \bint{c}$. By duality in linear programming, there exists a dual feasible vector $y$ satisfying
		\begin{eqnarray*}
			&&c_1^T x^+ - c_2^T x^- = b^T y,\\
			&&A x^+ - Ax^- \le b,\; x^+ \ge 0,\; x^- \ge 0,\\
			&&A^T y \le c_1,\; -A^T y \le -c_2,\; y \le 0.
		\end{eqnarray*}
		From the dual feasibility constraints, we have $A^T y = c_3$ for some $c_3 \in [c_2, c_1] \subseteq \bint{c}$. We will prove that $x = x^+ - x^-$ is optimal for the scenario with objective vector $c_3$ and right-hand-side vector $b$ in ILP~\eqref{eq:ilp:B}, by showing that it satisfies the system
		\begin{equation*}
		c_3^T x = b^T y,\; Ax \le b,\; A^T y = c_3,\; y \le 0.
		\end{equation*}
		Namely, it remains to show that $c_3^T (x^+ - x^-) = b^T y$ holds. Complementary slackness implies that for each $i \in \{1, \dots, n\}$ we have
		\begin{eqnarray*}
			&&(x^+)_i = 0\;\vee\;(c_1-A^T y)_i = 0, \text{ and}\\
			&&(x^-)_i = 0\;\vee\;(A^T y-c_2)_i = 0.
		\end{eqnarray*}
		If $(x^+)_i > 0$ and $(x^-)_i > 0$ for some index $i$, then we have $(c_1)_i = (c_2)_i = (c_3)_i$ by the choice of $c_3$. For $(x^+)_i = 0$ and $(x^-)_i > 0$ we have $(c_2)_i = (c_3)_i$, and analogically, $(c_1)_i = (c_3)_i$ for the symmetric case. By substituting into the equation
		\begin{equation*}
		c_1^T x^+ - c_2^T x^- = b^T y,
		\end{equation*}
		we can see that the desired constraint is satisfied in all cases. Therefore, $x^+ - x^-$ is an optimal solution of program~\eqref{eq:ilp:B}.
	\end{proof}
	
	Figure~\ref{fig:gen} provides an overview of the transformations preserving the feasible and optimal solution set of general interval linear programs, as presented in Section~\ref{ssec:gen}. Note that ILPs of type~\eqref{eq:ilp:C} are a special case of type~\eqref{eq:ilp:B} and the transformation of type~\eqref{eq:ilp:A} to type~\eqref{eq:ilp:B} can be obtained by transitivity. As shown in Figure~\ref{fig:spec}, all of the considered transformations preserve the optimal solution set for ILPs with a fixed coefficient matrix. Apart from the direct consequences of Theorem~\ref{thm:spec:AC} and Theorem~\ref{thm:spec:nonneg}, the remaining results follow again by transitivity.
	
	\begin{figure}[b]
		\begin{minipage}[t]{0.49\textwidth}
			\begin{tikzpicture}[scale=0.5]
			\tikzstyle{sipka}=[-{>[length=7pt, width=5pt]}]
			\tikzstyle{rect}=[draw,rectangle,minimum height=12ex,inner sep=3pt,text width=50pt, text height=2.25ex, align=center]
			\node[rect] (a) at (0,4) {Type (I): \small{$\bint{A} x=\bint{b}$, $x\ge 0$}};
			\node[rect] (b) at (-3.8,-1.2) {Type (II): \small{$\bint{A}x \le \bint{b}$}};
			\node[rect] (c) at (3.8,-1.2) {Type (III): \small{$\bint{A}x \le \bint{b}$, $x \ge 0$}};
			
			\draw[sipka,dashed] (a.south west) -- (b.north);
			%\draw[sipka] ([xshift=15pt]b.north) -- node[below,sloped] {transitivity} ([xshift=15pt]a.south west);
			\draw[sipka,dashed] (a.south east) -- (c.north);
			\draw[sipka] ([xshift=-15pt]c.north) -- ([xshift=-15pt]a.south east);
			%\draw[sipka] ([yshift=5pt]b.east) -- node[above] {\autoref{thm:spec:nonneg}} ([yshift=5pt]c.west);
			\draw[sipka] ([yshift=-5pt]c.west) -- ([yshift=-5pt]b.east);
			\node at (0,1) {$\min\,\bint{c}^T x$};
			\end{tikzpicture}
			\caption{Transformations preserving the feasible (dashed arrows) and optimal (solid arrows) solution set of a general interval linear program.}\label{fig:gen}
		\end{minipage}
		\hspace{0.02\textwidth}
		\begin{minipage}[t]{0.49\textwidth}
			\begin{tikzpicture}[scale=0.5]
			\tikzstyle{sipka}=[-{>[length=7pt, width=5pt]}]
			\tikzstyle{rect}=[draw,rectangle,minimum height=12ex,inner sep=3pt,text width=50pt, text height=2.25ex, align=center]
			\node[rect] (a) at (0,4) {Type (I): \small{$A x=\bint{b}$, $x\ge 0$}};
			\node[rect] (b) at (-3.8,-1.2) {Type (II): \small{$Ax \le \bint{b}$}};
			\node[rect] (c) at (3.8,-1.2) {Type (III): \small{$Ax \le \bint{b}$, $x \ge 0$}};
			
			\draw[sipka] (a.south west) -- (b.north);
			\draw[sipka] ([xshift=15pt]b.north) -- ([xshift=15pt]a.south west);
			\draw[sipka] (a.south east) -- node[above,sloped] {Thm.~\ref{thm:spec:AC}} (c.north);
			\draw[sipka] ([xshift=-15pt]c.north) --  ([xshift=-15pt]a.south east);
			\draw[sipka] ([yshift=5pt]b.east) -- node[above] {Thm.~\ref{thm:spec:nonneg}} ([yshift=5pt]c.west);
			\draw[sipka] ([yshift=-5pt]c.west) -- ([yshift=-5pt]b.east);
			\node at (0,1) {$\min\, \bint{c}^T x$};
			\end{tikzpicture}
			\caption{Transformations preserving the optimal solution set of an interval linear program with a fixed coefficient matrix.}\label{fig:spec}
		\end{minipage}
	\end{figure}
	
	\section{Optimal Values under Transformations}\label{sec:optval}
	Let us now discuss the effects of transformations on the set of all optimal values of an interval linear program and the optimal value range. Recall that the optimal value range refers to the interval $[\lb{f}, \ub{f}]$, which is the smallest interval that encloses all optimal values. This interval may differ from the set of optimal values, since not all values in the range have to be attained as optimal for a scenario. Again, we consider the two transformations that can possibly lead to a dependency problem: splitting equations into inequalities and substituting a difference of two non-negative variables for a free variable.
	
	\subsection{General Case}\label{ssec:gen:val}
	We have already seen in Section~\ref{ssec:gen} that splitting an equation into two opposite inequalities may change the optimal solution set of an interval linear program. This also holds for the optimal value range, namely, the transformation can change the worst-case bound ($\ub{f}$ for a minimization program). Moreover, this is caused not only by creating an infeasible scenario resulting in $\ub{f} = \infty$, but can also happen due to an expansion of the set of finite optimal values.
	\addtocounter{example}{-1}
	\begin{example}[continued]
		The optimal value range of the interval linear program
		\begin{equation*}
		\begin{array}{lr@{\ }l}
		\text{minimize } & \multicolumn{2}{l}{-x_1} \\
		\text{subject to } & [0,1]x_1 - x_2 = 0,\\
		&  x_2 \le 1,\\
		&  x_1, x_2 \ge 0.
		\end{array}
		\end{equation*}
		is the interval $(-\infty, -1]$. However, by splitting the equation $[0,1]x_1 - x_2 = 0$ into two inequalities, the solution $(0,0)$ becomes optimal and, thus, the value~$0$ belongs to the set of optimal values (and the optimal value range) of the transformed program. This shows that $\ub{f} = -1$ is no longer the worst-case optimal value.
	\end{example}
	
	On the other hand, we will now show that the transformation preserves the best-case optimal value~$\lb{f}$. The proof is a consequence of a result on a unified approach for computing the optimal value range by~\cite{Hladik:2009}. For the purposes of the following theorem, let us consider an interval linear program in the general form
	\begin{equation*}
	\begin{array}{lrrrr@{\ }l}
	\text{minimize } & \multicolumn{5}{l}{\bint{c_1}^T x_1 + \bint{c_2}^T x_2} \\
	\text{subject to } & \bint{A}x_1 &+ & \bint{B}x_2 & = & \bint{b_1},\\
	&  \bint{C}x_1 & + & \bint{D}x_2 & \le & \bint{b_2},\\
	& &  \multicolumn{2}{r}{x_1, x_2} &\ge & 0.
	\end{array}
	\end{equation*}
	Let $\bint{b} = (\bint{b_1}, \bint{b_2})$, $\bint{c} = (\bint{c_1}, \bint{c_2})$ and let $\mathcal{M}$ denote the set of all feasible solutions. Furthermore, let $\mathcal{N}$ be the weakly feasible set of the dual ILP. Then, we can apply the formulas presented in Theorem~\ref{thm:hladik:f} to compute the optimal value range.
	
	\begin{theorem}[\cite{Hladik:2009}]\label{thm:hladik:f}
		We have
		\begin{equation}\label{eq:lb}
		\lb{f} = \inf \{c_c^T x - c_\Delta^T \lvert x \rvert : x \in \mathcal{M}\}.
		\end{equation}
		If $\ub{f} < \infty$, then
		\begin{equation}\label{eq:ub}
		\ub{f} = \sup \{b_c^T y + b_\Delta^T \lvert y \rvert : y \in \mathcal{N}\}.
		\end{equation}
	\end{theorem}
	
	Theorem~\ref{thm:optval:AC} proves that the transformation of splitting an equation into two opposite inequalities does not change the best-case optimal value of an interval linear program. The result implies that it is possible to convert an equation-constrained ILP of type~\eqref{eq:ilp:A} into an inequality constrained ILP of type~\eqref{eq:ilp:C}, while preserving the best optimal value.
	
	\begin{theorem}\label{thm:optval:AC} 
		Transforming $\bint{A}x = \bint{b}$ into $\bint{A}x \le \bint{b}, \bint{A}x \ge \bint{b}$ does not change the best optimal value $\lb{f}$ of a minimization ILP.
	\end{theorem}
	\begin{proof}
		By Theorem~\ref{thm:hladik:f}, the best optimal value $\lb{f}$ of a general interval linear program can be found by minimizing the fixed objective function $c_c^T x - c_\Delta^T \lvert x \rvert$ over the set of all weakly feasible solutions. Since the applied transformation of splitting an equation into two inequalities does not change the weakly feasible set of an interval system (see Theorem~\ref{thm:Li}), the value of $\lb{f}$ remains the same.
	\end{proof}
	
	By the property of strong duality in classical linear programming, we can also derive analogous results for the second considered transformation. For this case, let us first show that substituting the difference of two non-negative variables for a free variables can change the best-case bound $\lb{f}$, as well as the set of optimal values.
	
	\begin{example}
		Consider the dual interval linear program to~\eqref{eq:ex}, which can be rewritten into a minimization form as
		\begin{equation}\label{eq:ex:2}
		\begin{array}{lrrrr@{\ }l}
		\text{minimize } & \multicolumn{5}{l}{-y_2} \\
		\text{subject to } & [0,1]y_1 & & & \le & -1,\\
		&  -y_1 & + & y_2 & \le & 0,\\
		& &  \multicolumn{2}{r}{y_2} &\le & 0.
		\end{array}
		\end{equation}
		The optimal value range of program~\eqref{eq:ex:2} is the interval $[1, \infty)$. Let us now substitute the term $y_1^+-y_1^-$ with $y_1^+, y_1^- \ge 0$ for the free variable $y_1$. Analogously to the previous example, the set of optimal values changes and the best-case bound $\lb{f} = 1$ is no longer valid (again, the value $0$ becomes optimal).
	\end{example}
	
	However, we can show that the transformation preserves the worst-case bound $\ub{f}$. The proof uses the notion of strong feasibility, which also provides a characterization of finiteness of the bound $\ub{f}$, as stated in Theorem~\ref{thm:strfeas}. An interval linear program is said to be \emph{strongly feasible}, if each scenario of the program is feasible.
	
	\begin{theorem}[\cite{Hladik:2009}]\label{thm:strfeas}
		An interval linear program (in the general form) is strongly feasible if and only if $\ub{f} < \infty$.
	\end{theorem}
	
	\begin{theorem}
		Substituting $x = x^+ - x^-$ with $x^+, x^- \ge 0$ for a free variable $x$ does not change the worst optimal value $\ub{f}$ of a minimization ILP.
	\end{theorem}
	\begin{proof}
		Let us denote by $\ub{f}$ and $\ub{f}^\pm$ the worst-case optimal value of the original program and the transformed program created by the introducing the substitution, respectively. 
		
		First, assume that $\ub{f} = \infty$. Since all of the original scenarios are also included in the transformed ILP, the latter is a relaxation of the original program and $\ub{f} \le \ub{f}^\pm$ holds. Therefore, we also have $\ub{f}^\pm = \infty$.
		
		Further, let $\ub{f} < \infty$ hold for the original interval program. Then, the program is strongly feasible. Since the conditions for strong feasibility of the original and the transformed program are equivalent for both equation and inequality constraints (see \cite{Hladik:2017}), the latter is also strongly feasible and,  by Theorem~\ref{thm:strfeas}, the property $\ub{f}^\pm < \infty$ holds. By formula~\eqref{eq:ub} of Theorem~\ref{thm:hladik:f} we can calculate the worst-case optimal values by optimizing the objective function $b_c^T y + b_\Delta^T \lvert y \rvert$ over the dual feasible set of the respective ILPs. Note that applying the substitution to a free variable in the primal ILP corresponds to splitting an equation into two opposite inequalities in the dual ILP. As this preserves the set of feasible solutions (Theorem~\ref{thm:Li}), the dual feasible sets of the original and the transformed program are equal, and thus $\ub{f} = \ub{f}^\pm$. 
	\end{proof}
	
	\subsection{Special Case: Fixed Coefficient Matrix}\label{ssec:spec:val}
	In Section~\ref{ssec:gen:val} we have seen that the optimal values and bounds of the optimal value range may change under some transformations. Therefore, it is natural to ask whether these properties are preserved at least for ILPs with a fixed coefficient matrix. Unfortunately, even for this special class of programs, the transformations may change the optimal value range, as shown by the following trivial examples.
	
	\begin{example}\label{ex:optval:fixed}
		Consider the following interval linear programs:
		\vskip\abovedisplayskip
		\noindent
		\begin{subequations}
			\begin{minipage}{.5\linewidth}
				\begin{equation}\label{eq:ex:3a}
				\begin{array}{lrrrr@{\ }l}
				\text{minimize } & \multicolumn{5}{l}{[0,1]x} \\
				\text{subject to } & x & \ge & 1,\\
				\end{array}
				\end{equation}
			\end{minipage}%
			\begin{minipage}{.5\linewidth}
				\begin{equation}\label{eq:ex:3b}
				\begin{array}{lrrrr@{\ }l}
				\text{minimize } & \multicolumn{5}{l}{-y} \\
				\text{subject to } & y & = & [0,1].\\
				\end{array}\qquad
				\end{equation}
			\end{minipage}
		\end{subequations}
		\vskip\belowdisplayskip\noindent
		The optimal value range of ILP~\eqref{eq:ex:3a} is the interval $[0,1]$, for ILP~\eqref{eq:ex:3b} it is the opposite interval $[-1,0]$. By substituting $x^+-x^-$ with non-negative variables $x^+, x^-$ for the free variable $x$ in \eqref{eq:ex:3a}, we introduce an unbounded scenario (setting the objective to $0x^+ - 1x^-$) and the best-case bound of the optimal value range changes to $\lb{f} = -\infty$.
		Similarly, by splitting the equation in \eqref{eq:ex:3b} into two opposite inequalities $y \le [0,1], y \ge [0,1]$, we create an infeasible scenario leading to $\ub{f} = \infty$.
	\end{example}
	
	However, note that there is an important difference between Example~\ref{ex:optval:fixed} and the previous examples with an interval coefficient matrix. While in the examples of Section~\ref{ssec:gen:val} we have seen that a transformation may cause a change in the set of optimal values, in programs \eqref{eq:ex:3a} and \eqref{eq:ex:3b} the transformations only change one of the bounds in the optimal value range due to infeasibility or unboundedness of a newly introduced scenario.
	
	Let us now consider the set of all finite optimal values of an interval linear program with a fixed coefficient matrix. The following theorems prove that even though the optimal value range may still change when transforming a~program with a fixed matrix, this can only be caused by the infinite optimal values for infeasible or unbounded scenarios and the finite optimal values remain the same.
	
	\begin{theorem}\label{thm:spec:optval}
		Transforming $Ax = \bint{b}$ into $Ax \le \bint{b}, Ax \ge \bint{b}$ does not change the set of all finite optimal values of an ILP with a fixed coefficient matrix.
	\end{theorem}
	\begin{proof}
		Clearly, all optimal values of the original program remain optimal in the transformed program. Let a solution $x^*$ be optimal for a scenario given by $c \in \bint{c}$ and $b_1, b_2 \in \bint{b}$ in the transformed program. It is easy to see that $x^*$ is also optimal for the scenario determined by $c$ and $b_3 = Ax^*$ of the original program, since it has the same objective function and a restricted feasible set. Therefore, the optimal value $c^T x^*$ of the transformed ILP is also optimal for the original ILP.
	\end{proof}
	
	\begin{theorem}
		Substituting $x = x^+ - x^-$ with $x^+, x^- \ge 0$ for a free variable $x$ does not change the set of all finite optimal values of an ILP with a fixed coefficient matrix.
	\end{theorem}
	\begin{proof}
		By strong duality of linear programming, the sets of finite optimal values of an interval linear program and its dual are the same. As introducing the substitution in the primal ILP corresponds to the transformation of splitting an equation into two inequalities in the dual ILP, applying Theorem~\ref{thm:spec:optval} yields the result.
	\end{proof}
	
	\section{Conclusion}\label{sec:concl}
	We addressed the dependency problem in transforming interval linear programs using the techniques known from classical linear programming. We showed that while it is possible to switch the objective of an interval linear program or add slack variables to convert inequalities into equations, other transformations are not always applicable to interval programs without affecting some of their properties.
	
	Namely, we considered three commonly used forms of interval linear programs. It was shown that the set of all optimal solutions may change, in general, under the transformations among these forms. Therefore, we also studied a special class of interval programs with a fixed coefficient matrix, for which we proved that all of the transformations preserve the optimal set.
	
	Furthermore, we also studied the effect of the transformations on the set of optimal values and the optimal value range of an interval linear program. We proved that the best-case optimal value $\lb{f}$ of a minimization program remains the same when an equation is split into two opposite inequalities, while the worst-case optimal value $\ub{f}$ is preserved when substituting the difference of two non-negative variables for a free variable. The complementary results do not hold, even in the case of a fixed coefficient matrix. However, the set of all finite optimal values does not change for transformations on the special class of programs.
	
	The results allow us to generalize the theory that was derived for a particular form of ILPs to all other types that can be obtained by a transformation respecting the studied properties. We believe that they also provide a better insight into the nature of the dependency problem in interval optimization.

	%\begin{acknowledgements}
	%\end{acknowledgements}
	
\bibliographystyle{abbrv}      % basic style, author-year citations
	\bibliography{sor2017}   % name your BibTeX data base
	
\end{document}